\documentclass[reqno, 12pt]{amsart}
\pdfoutput=1
\makeatletter
\let\origsection=\section \def\section{\@ifstar{\origsection*}{\mysection}} 
\def\mysection{\@startsection{section}{1}\z@{.7\linespacing\@plus\linespacing}{.5\linespacing}{\normalfont\scshape\centering\S}}
\makeatother        

\usepackage{amsmath,amssymb,amsthm}
\usepackage{mathrsfs}
\usepackage{mathabx}\changenotsign
\usepackage{dsfont}
 
\AtBeginDocument{%
   \def\MR#1{}
} 

\usepackage[dvipsnames]{xcolor}
\usepackage[backref]{hyperref}
\hypersetup{
    colorlinks,
    linkcolor={red!60!black},
    citecolor={green!60!black},
    urlcolor={blue!60!black}
}

\usepackage{graphicx}

\usepackage[open,openlevel=2,atend]{bookmark}

\usepackage[abbrev,msc-links,backrefs]{amsrefs} 
\usepackage{doi}

\renewcommand{\PrintDOI}[1]{\doi{#1}}

\usepackage[T1]{fontenc}
\usepackage{lmodern}
\usepackage[babel]{microtype}
\usepackage[english]{babel}

\usepackage{geometry}
\numberwithin{equation}{section}
\numberwithin{figure}{section}

\usepackage{enumitem}
\def\rmlabel{\upshape({\itshape \roman*\,})}

\let\polishlcross=\l
\def\l{\ifmmode\ell\else\polishlcross\fi}

\def\paragraph#1{%
  \noindent\textbf{#1.}\enspace}

\let\emptyset=\varnothing
\let\setminus=\smallsetminus

\makeatletter
\def\moverlay{\mathpalette\mov@rlay}
\def\mov@rlay#1#2{\leavevmode\vtop{   \baselineskip\z@skip \lineskiplimit-\maxdimen
   \ialign{\hfil$\m@th#1##$\hfil\cr#2\crcr}}}
\newcommand{\charfusion}[3][\mathord]{
    #1{\ifx#1\mathop\vphantom{#2}\fi
        \mathpalette\mov@rlay{#2\cr#3}
      }
    \ifx#1\mathop\expandafter\displaylimits\fi}
\makeatother

\newcommand{\rbrac}[1]{\left( #1 \right)}

\DeclareFontFamily{U}  {MnSymbolC}{}
\DeclareSymbolFont{MnSyC}         {U}  {MnSymbolC}{m}{n}
\DeclareFontShape{U}{MnSymbolC}{m}{n}{
    <-6>  MnSymbolC5
   <6-7>  MnSymbolC6
   <7-8>  MnSymbolC7
   <8-9>  MnSymbolC8
   <9-10> MnSymbolC9
  <10-12> MnSymbolC10
  <12->   MnSymbolC12}{}
\DeclareMathSymbol{\powerset}{\mathord}{MnSyC}{180}

\usepackage{tikz}
\usetikzlibrary{calc,decorations.pathmorphing}
\pgfdeclarelayer{background}
\pgfdeclarelayer{foreground}
\pgfdeclarelayer{front}
\pgfsetlayers{background,main,foreground,front}

\let\epsilon=\varepsilon
\let\eps=\epsilon
\let\rho=\varrho
\let\theta=\vartheta
\let\kappa=\varkappa

\let\E=\EE

\newcommand{\tbf}[1]{\textbf{#1}}
\newcommand{\mc}[1]{\mathcal{#1}}

\def\rev{\mathrm{rev}}
\def\Rev{\mathrm{Rev}}

\newcommand{\bin}{\mathrm{Bin}}

\theoremstyle{plain}
\newtheorem{thm}{Theorem}[section]

\newtheorem{prop}[thm]{Proposition}
\newtheorem{claim}[thm]{Claim}

\newtheorem{cor}[thm]{Corollary}
\newtheorem{lemma}[thm]{Lemma}

\theoremstyle{definition}

\usepackage{accents}

\let\phi=\varphi

\title{A note on non-isomorphic edge-color classes in random graphs}

\author{Patrick Bennett}
\address{Department of Mathematics, Western Michigan University, Kalamazoo, MI, USA}
\email{patrick.bennett@wmich.edu}
\thanks{The first author was supported in part by Simons Foundation Grant \#426894.}

\author{Ryan Cushman}
\address{Department of Mathematics, Toronto Metropolitan University, Toronto, ON, Canada}
\email{ryan.cushman@ryerson.ca}

\author{Andrzej Dudek}
\address{Department of Mathematics, Western Michigan University, Kalamazoo, MI, USA}
\email{andrzej.dudek@wmich.edu}
\thanks{The third author was supported in part by Simons Foundation Grant \#522400.}

\author{Elizabeth Sprangel}
\address{Department of Mathematics, Iowa State University, Ames, IA, USA}
\email{sprangel@iastate.edu}

\begin{document}

\maketitle

\begin{abstract}
    For a graph $G$, let $\tau(G)$ be the maximum number of colors such that there exists an edge-coloring of $G$ with no two color classes being isomorphic. We investigate the behavior of $\tau(G)$ when $G=G(n, p)$ is the classical Erd\H{o}s-R\'enyi random graph. 
\end{abstract}

\section{Introduction}

Recall that for graphs $H, G$ we define an \tbf{$H$-factor of $G$} to be a covering of the vertices of $G$ by $|G|/|H|$ vertex-disjoint copies of $H$. Equivalently, an $H$-factor of a graph $G$ is a partition of the vertices of $G$ such that each part contains $H$ as a spanning subgraph.
For example, for $K_r$-factors, we wish to partition the vertices of $G$ such that each part is an $r$-clique. 
In this paper we consider what is in some sense an opposite problem:
when can we partition the edges of $G$ into parts which are pairwise-non-isomorphic?  Clearly, using one part will be sufficient, but letting each edge be its own part would not. Therefore, we are interested in the maximum number of parts where this property holds. More precisely, we define the parameter $\tau(G)$ to be the maximum number of colors such that there exists a coloring of the edges of $G$ such that any two color classes are  non-isomorphic. Note that $\tau(G)$ exists, since one color satisfies the condition. In this paper we focus on $\tau(G(n, p))$ where $G(n, p)$ is the Erd\H{o}s-R\'enyi random graph. 

Before we state our results we give some history for $H$-factors in $G(n, p)$. In~\cite{E1981}, Erd\H{o}s reported that the Shamir problem~\cite{SS1983}---when do perfect matchings occur in the $r$-uniform random hypergraph $H_r(n,p)$?---was one of the combinatorial problems he would most like to see solved. 
A similar question was proposed and some initial work done by Ruc\'inski~\cite{R1992} and Alon and Yuster~\cite{AY1993} for the case of $G(n,p)$: what is the threshold for an $H$-factor? 
After some partial results,
both questions were answered by Johansson, Kahn and Vu in~\cite{JKV2008}.
More precise answers have also recently been offered by Kahn~\cite{K2019} and Riordan~\cite{R2018}.
Kahn determined a sharp threshold for a matching in $H_r(n,p)$ and Riordan created a coupling between $H_r(n,p)$ and $G(n,p)$ in which hyperedges of $H_r(n,p)$ correspond to cliques in $G(n,p)$ for $r \ge 4$ (the case of $r =3$ was proved in~\cite{H2021} by Heckel).

We have several results about $\tau(G(n,p))$, covering various edge densities. Our first theorem gives bounds for a large range of $p$. We say that a sequence of events $E_n$ happens \tbf{asymptotically almost surely} (a.a.s.) if $\Pr(E_n)\rightarrow 1$ as $n \rightarrow \infty$. All asymptotics in this paper are as $n \rightarrow \infty$, and $p=p(n)$ may depend on $n$. We sometimes write inequalities that are valid only for sufficiently large $n$. We use the standard big-$O$ and little-$o$ notation, as well as the standard $\Omega, \Theta$ notation. Our first result covers the dense and not-too-sparse regime.

\begin{thm}\label{thm:main}
There is an absolute constant $C>0$ such that if $\frac{C\log n}{n} \le p\le \frac{1}{\log n}$ and $G=G(n,p)$, then  a.a.s.
\[
\Omega\left(\frac{pn^2}{\log n}\right) = \tau(G) = O\left(\frac{pn^2\log\log n}{\log n}\right).
\]
\end{thm}

When we let $G(n,p)$ become slightly sparser, we can also determine bounds for $\tau$, although the lower bound is slightly less precise. 
\begin{thm}\label{thm:main_sparse}
There is an absolute constant $C>0$ such that if $p\ge \frac{C}{n}$ and $G=G(n,p)$, then  a.a.s.
\[
\Omega\left(\frac{pn^2}{\log^2 n}\right) = \tau(G) = O\left(\frac{pn^2\log\log n}{\log n}\right).
\]
\end{thm}
Letting $p=1$ yields the following result for complete graphs.

\begin{cor}\label{cor:kn} Let $K_n$ denotes the complete graph of order~$n$. Then,
\[
\Omega\left(\frac{n^2}{\log^2 n}\right) = \tau(K_n) = O\left(\frac{n^2\log\log n}{\log n}\right).
\]
\end{cor}

Finally, we consider some very sparse regimes, in which we determine the order of magnitude of $\tau(G(n,p))$.
\begin{thm}\label{thm:main_very_sparse}
Let $G=G(n,p)$ be such that $n^{-\frac{k}{k-1}} \ll p \ll n^{-\frac{k+1}{k}}$ for some positive integer~$k \ge 2$. Let 
$\ell = \ell(k) = \lfloor\frac{\sqrt{8k-7}+1}{2}\rfloor$. Then,
\[
\tau(G) = \Theta\left(n^{\frac{(\ell+2)(\ell-1)}{2\ell}} p^{\frac{\ell-1}{2}} \right).
\]
\end{thm}

In the sparsest cases (Theorem \ref{thm:main_very_sparse}), we are able to determine the order of magnitude while in the other cases (Theorems \ref{thm:main} and \ref{thm:main_sparse}), the upper and lower bounds are separated by a logarithmic factors in the numerator or denominator. It would be an interesting future direction to remove this ambiguity about the order of magnitude and have a more precise result. In particular improving Corollary \ref{cor:kn} may be easier than working with $G(n, p)$.

We will be using the following forms of Chernoff's bound (see, e.g., \cite{JLR}). 

\begin{lemma}[Chernoff bound]
Let $X\sim \bin(n,p)$ and $\mu = \E(X)$. Then, for all $0<\delta<1$
\begin{equation*}
\Pr(X \ge (1+\delta) \mu) \le \exp(-\mu \delta^2/3)
\end{equation*}
and
\begin{equation*}
\Pr(X \le (1-\delta)\mu) \le \exp(-\mu \delta^2/2). 
\end{equation*}
Furthermore if $R \ge 2e \mathbb \mu$, we have 
\begin{equation*}
\Pr(X \ge R) \le 2^{-R}.
\end{equation*}
\end{lemma}

We recently learned that the lower bounds in Theorems~\ref{thm:main} and~\ref{thm:main_sparse} can be also derived from a result of Ferber and Samotij~\cite{FS2019}. Since their approach is different from ours, we decided to keep our proofs for the sake of completeness.

\section{Proof of Theorem~\ref{thm:main}}

\subsection{Upper bound}
We start with the lower bound. Let $p\ge\frac{\log n}{n}$ and $G=G(n,p)$. We are going to show that $\tau(G) \le pn^2/f$, where $f := f(n) = \frac{\log n}{2\log \log n}$. 

For a proof by contradiction, assume that $\tau(G) > pn^2/f$. Let $\alpha_{\ge f}$ be the number of color classes with at least $f$ edges.
Similarly, we define $\alpha_{<f}$ as the number of color classes with less than $f$ edges. 

First we estimate $\alpha_{<f}$ from above. Observe that any graph with at most $i$ edges and no isolated vertices cannot have more than $2i$ vertices. 
Thus the number of colors having $i$ edges is at most $\displaystyle \binom{\binom{2i}{2}}{i}$, the number of ways to choose $i$ edges on $2i$ vertices. Since we consider graphs with at most $f$ edges, we get
\begin{align}
\alpha_{<f} &\le \sum_{i=1}^f \binom{\binom{2i}{2}}{i} \le f\binom{\binom{2f}{2}}{f}
\le f\binom{2f^2}{f}
\le f(2ef)^f
=f \exp\{f\log(2ef)\}\notag \\
&=f \exp\left\{ \frac{\log n}{2\log \log n} \left( \log\log n - \log\log\log n +1 \right) \right\}\notag \\
& \le f \exp\left\{\frac{1}{2}\log n\right\} = fn^{1/2} = o\left(\frac{pn^2}{f}\right),\label{eq:ub}
\end{align}
where the latter follows from $p\ge\frac{\log n}{n}$. 

Next note that $\alpha_{\ge f} \ge 2pn^2/(3f)$; indeed, otherwise we would have $\alpha_{\ge f} < 2pn^2/(3f)$ and therefore
\[
\frac{pn^2}{f} < \tau(G) = \alpha_{\ge f} + \alpha_{< f} < \frac{2pn^2}{3f} + o\left(\frac{pn^2}{f}\right) < \frac{pn^2}{f},
\]
a contradiction.

Finally, observe that the the lower bound on $\alpha_{\ge f}$ implies that the number of edges in $G$ is at least 
\[
f \alpha_{\ge f} \ge f \frac{2pn^2}{3f} = \frac{2pn^2}{3},
\]
which is a contradiction, since it is well-known that the number of edges in $G$ is highly concentrated around its mean $\binom{n}{2}p=(1+o(1)) pn^2/2$. Thus, $\tau(G) \le pn^2/f$, as required.


\subsection{Lower bound}

Let $\frac{C\log^2 n}{n} \le p\le \frac{1}{\log n}$ for some $C>0$ and $G = G(n,p)$. 
We start with an auxiliary result. 
\begin{prop}\label{prop:tree}
Let $t \ge 1$. There is a set $\mathcal{T}$ of at least 
\[
\frac{2^{t-1}}{t^{3/2}(t+1)}
\]
pairwise-non-isomorphic trees on $t$ vertices with maximum degree at most three.
\end{prop}

\begin{proof}

First we will enumerate all \emph{rooted binary trees} on $t$ vertices. 
Designate one vertex as the \emph{root}. We now choose zero, one, or two vertices that will be adjacent to the root and assign these as the right and left \emph{ children} of the root (where applicable). For each new vertex, repeat this process of adding at most two children  until $t$ vertices have been used. 
It is not difficult to see that the number of all rooted binary trees is exactly the $t$-th Catalan number, i.e., $\frac{1}{t+1}\binom{2t}{t}$ (see, e.g.,~Section 2.3.4.4 in~\cite{Knuth1997}). 

Now observe that at most $t2^t$ rooted binary trees can be in any particular isomorphism class. Indeed, let $T$ be a binary tree of order $t$ without a marked root. We need to find out how many rooted binary trees are isomorphic to $T$. Since $T$ has order $t$, we have at most $t$ choices for a root. Then each of the other $t-1$ vertices in $T$ will have a unique parent, with at most two choices for how each vertex relates to its parent (left or right), yielding at most  $t2^t$ trees isomorphic to $T$.

Thus, the number of distinct isomorphism classes of these binary trees is at least 
\[
\frac{1}{t+1}\binom{2t}{t} \cdot \frac{1}{t2^t} \ge \frac{2^{2t-1}}{t(t+1)2^t\sqrt{t}} = \frac{2^{t-1}}{t^{3/2}(t+1)},
\]
since $\binom{2t}{t} \ge \frac{2^{2t-1}}{\sqrt{t}}$ for sufficiently large~$t$.
\end{proof}

Let $\mathcal{T}$ be a set of pairwise-non-isomorphic trees of order $t=\lfloor 3\log_2 n \rfloor$ and maximum degree at most~$3$ guaranteed by Proposition~\ref{prop:tree}. Then 
\[
|\mathcal{T}| \ge \frac{2^{t-1}}{t^{3/2}(t+1)} \ge n^2. 
\]

We now describe a randomized algorithm \texttt{GREEDY\_EMBED} that (if successful) embeds the trees from $\mathcal{T}$ into $G$ one by one. Our description (and analysis) of the algorithm will involve \emph{revealing} the pairs of vertices in $G=G(n, p)$. When we \emph{reveal} a pair $uv$ we are determining whether $uv$ is an edge (which of course happens with probability $p$). We initialize our algorithm with a set of $n$ vertices where no pairs have been revealed yet. As the algorithm runs we reveal pairs and embed trees from $\mc{T}$ using the edges we find. 

Let $T$ be a tree from $\mathcal{T}$. We will define an embedding $g_T: V(T) \rightarrow V(G)$. Clearly, we can order its vertices as $u_1,\dots,u_t$ in such a way that the subgraph of $T$ induced by $u_1,\dots,u_j$ is connected for all $1\le j\le t$.  Then set $g_T(u_1) := v_{1}$ for some vertex $v_{1}$ chosen uniformly at random. We call $v_1$ the \textit{active vertex} and all other vertices of $G$ are \textit{inactive}. We find the image of $u_2$ by revealing pairs $v_{1}v$ uniformly at random. We designate each step of this process as the revealing of another pair. Denote
$$
\Rev_j(u) = \lbrace v \in V(G): uv \text{ has been revealed in step at most } j\rbrace
$$  
and let $\rev_j(u) = |\Rev_j(u)|$.
For each new vertex $v$ added to $\rev_j(v_{1})$, we perform a Bernoulli trial with probability $p$ of successfully finding the edge $v_{1}v$ in $G$. As soon as we find one pair $v_{1}v$ that is an edge, we set $v_{2} := v$ and $g_T(u_2) := v_{2}$. At that point we will start trying to embed $u_3$, which (by our choice of the ordering $u_1, \ldots u_t$) is a neighbor of say $x \in \{u_1, u_2\}$. Now we say the \textit{active vertex} is $g_T(x)$. In general the active vertex is the image of the unique already-embedded neighbor of the vertex we are currently trying to embed. 
We continue this greedy process until $T$ is embedded. Next we choose another tree from $\mc{T}$ and repeat the embedding process, starting with a vertex chosen uniformly at random. We keep embedding trees this way as long as we can. The only way this can fail is if we reach a point where we have revealed all pairs containing the active vertex. The next proposition will imply that  we will be able to embed $\Omega(\frac{n^2p}{\log n})$ trees. 

\begin{prop}
Let $\epsilon>0$ be a sufficiently small constant.
Then in \texttt{GREEDY\_EMBED} on $G$, every vertex $v$ has the following a.a.s.:
\begin{enumerate}[label=\rmlabel]
\item \label{prop1} we find $(1+o(1))\eps n^2 p$ edges in $G$ after $\eps n^2$ steps, and
\item \label{prop2} $\rev_{\epsilon n^2}(v) \le 1000\epsilon n$.
\end{enumerate}
\end{prop}
Indeed, this proposition tells us that we will be able to continue the embedding process for $\epsilon n^2$ steps. 
In addition, this tells us that we may embed $(1+o(1))\epsilon n^2 p/(t-1) = \Omega(\frac{n^2p}{\log n})$ trees. To obtain the lower bound on $\tau(G)$, we color each of these trees with a different color and the remaining edges with one extra color. The trees were chosen to be pairwise-non-isomorphic, and the other color class a.a.s. has about $\left( \binom n2 - \eps n^2\right) p$ edges which is way too many to be isomorphic to any of the trees. Thus we are done with the lower bound once we prove the proposition. 

\begin{proof}
Assuming the process does not fail in the first $\eps n^2$ steps, the number of edges found during those steps is distributed as $\bin(\eps n^2, p)$. An easy application of the Chernoff bound gives that this random variable is a.a.s. $(1+o(1))\eps n^2 p$. Assuming \ref{prop2} holds, the process does not fail in the first $\eps n^2$ steps. Thus \ref{prop1} will follow from \ref{prop2}.

We turn to proving \ref{prop2}. Fix $\eps = 1/10000$ and $C = 4 / \eps$. We define the following:
$$
\Rev_j^+(v) = \lbrace w \in V(G): vw \text{ has been revealed in step at most $j$ while $v$ is active}\rbrace
$$
$$
\Rev_j^-(v) = \lbrace w \in V(G): vw \text{ has been revealed in step at most $j$ while $v$ is inactive}\rbrace.
$$
Similarly, we let $\rev_j^{-}(v) = |\Rev_j^{-}(v)|$ and $\rev_j^{+}(v) = |\Rev_j^{-}(v)|$.
Thus 
$$
\rev_j(v) = \rev_j^-(v) + \rev_j^+(v).
$$
For a fixed vertex $v$, we would like to bound the number of times $v$ becomes active. $v$ can become active in two ways: $v$ could be the first vertex of a new tree, or $v$ could be some other vertex in a tree. The first vertex of each tree is chosen uniformly at random (independent of everything that has happened in the process previously), so we might as well make those choices before the process runs. We already know that a.a.s. we will find at most $(1+o(1))\eps n^2 p \le 2 \eps n^2 p$ edges (``at most'' since the process might fail early), and so we will embed at most $2 \eps n^2 p /(t-1) \le b:= 2 \eps n^2 p / (5 \log n)$ trees. So we will choose $b$ vertices, say $w_1, \ldots w_b$ independently and uniformly at random (and with replacement). $w_k$ will be the first vertex of the $k$th tree (assuming the process makes it to the $k$th tree before step $\eps n^2$. We chose more vertices $w_k$ than we will need but this is OK). The number of times our fixed vertex $v$ is chosen as some vertex $w_k$ is distributed as $\bin(b, 1/n)$ which has expectation  $\frac{b}{n} = 2 \eps n p / (5 \log n)$. Using Chernoff, the probability that $v$ is chosen here $\eps np / \log n$ times (i.e. more than twice the expected number of times) is at most $\exp\rbrac{-\frac 13 \cdot \frac{2 \eps n p}{ 5 \log n}} \le \exp\rbrac{-\frac {2}{15} C \log n} =o\rbrac{\frac 1n}$ where the last equality easily follows from our choice of $C$. Thus by the union bound over all vertices $v$, a.a.s. each vertex appears at most $\eps np / \log n$ times as the first vertex of a tree. Thus we have a bound on one of the ways $v$ can become active this way. 

Now we bound the number of times $v$ becomes active in the other way (when we have already embedded at least one vertex of some tree, the current active vertex is say $u \neq v$, the process chooses the pair $uv$ to reveal, and an edge is found). To do this we will define some ``good events'' which we will use for the rest of the proof (we will not use all parts of this definition yet but we will later). For $1 \le i \le \eps n^2$, we define
\begin{itemize}
\item $\mathcal{E}_1(i)$ is the event that, for every vertex $v$, $\rev_i^+(v)\le 500 \eps n$
\item $\mathcal{E}_2(i)$ is the event that, for every vertex $v$, $\rev_i^-(v)\le 500 \eps n$
\item $\mathcal{E}(i) = \mathcal{E}_1(i) \cap \mathcal{E}_2(i)$.
\end{itemize}
(Note that we will be done with \ref{prop2} when we show that $\mathcal{E}(\eps n^2)$ holds a.a.s.) If $\mathcal{E}(i)$ holds and $v$ is not active, the probability that $v$ becomes the active vertex at step $i+1$ is at least $p/n$ and at most $p/(n-1000 \eps n)$. To handle the possibility that $\mathcal{E}(i)$ does not hold for some $i \le \eps n^2$, we can imagine running a ``phantom process'' which just chooses a sequence of vertices to call ``active'' at steps $i+1, \ldots \eps n^2$. In this phantom process we will say the probability that $v$ is the ``active'' vertex at any given step is $p/(n-1000 \eps n)$. The expected number of times $v$ becomes active this way  (either in the real process or the phantom process) is then between $\eps n p$ and $\eps n p / (1-1000 \eps)$. Chernoff gives us that a.a.s. each vertex $v$ becomes active this way at most $2 \eps np$ times. Combined with the other way to become active (in the previous paragraph), each vertex $v$ becomes active at most $3 \eps np$ times.

Now we show that $\mathcal{E}_1(\eps n^2)$ holds a.a.s. Since $v$ becomes active at most $3 \eps np$ times and each tree we embed has maximum degree 3, the process will need to find at most $9 \eps np$ edges emanating from $v$. The probability that out of $500 \eps n$ independent trials with success probability $p$, we find fewer than $9 \eps np$ successes is at most (using Chernoff) 
\[
2^{-9 \eps np}  = \exp\{-\Omega(\log^2 n)\} = o\rbrac{\frac 1n}.
\]
Now by the union bound over vertices $v$, $\mathcal{E}_1(\eps n^2)$ holds a.a.s.

We turn to showing that $\mathcal{E}_2(\eps n^2)$ holds a.a.s. If $\mathcal{E}_2(\eps n^2)$ fails (and $\mathcal{E}_1(\eps n^2)$ holds), then there exists a vertex $v$ and a step $j \le \eps n^2$ such that $\mc{E}(j-1)$ holds but $\mc{E}_2(j)$ fails, i.e. $\rev_j^-(v) \ge 500\eps n$. Thus, at step a step $i < j$ where $v$ is inactive, for some active vertex $w$ to pick the pair $wv$ to reveal has probability at most 
$$
\frac{1}{n-1-\rev_i(v)} \le \frac{1}{n-1-500\eps n}.
$$
Therefore, we may write
\begin{align*}
\Pr(\mathcal{E}_2^c(j) | \mathcal{E}(j-1)) &\le n\binom{j}{500\eps n} \left(\frac{1}{n-1-500\eps n}\right)^{500\eps n}
\le n\binom{\epsilon n^2}{500\eps n} \left(\frac{1}{n-1-500\eps n}\right)^{500\eps n}\\
&\le n\left( \frac{e\epsilon n}{500\eps}\right)^{500\eps n} \left( \frac{1}{(1-500\eps + 1/n)n}\right)^{500\eps n}
\le  n \left( \frac{e\epsilon}{500\eps(1-500\eps)} \right)^{500\eps n}
\end{align*}
and since $e\varepsilon/500\eps(1-500\eps) < 1$ for sufficiently small $\varepsilon$, we have
\[
\Pr(\mathcal{E}_2^c(\epsilon n^2)) \le \Pr\left( \exists j\le \epsilon n^2 : \mathcal{E}_2^c(j) | \mathcal{E}(j-1) \right) 
\le \epsilon n^3 \left( \frac{e\epsilon}{500\eps(1-500\eps)} \right)^{500\eps n} = o(1).
\]
\end{proof}

\section{Proof of Theorem~\ref{thm:main_sparse}}

\subsection{Upper bound}
The proof is basically the same as the proof of the upper bound of Theorem~\ref{thm:main}, since the inequality~\eqref{eq:ub} is also valid for $p=\Omega(\frac{1}{n})$.

\subsection{Lower bound}
First recall that a \emph{partition} of a positive integer $k$ is a way of writing $k$ as a sum of positive integers. Let $p(k)$ be the number of partitions of $k$. It is known (see, e.g,~\cite{A1988}) that as $k$ grows,
$\log(p(k)) =(1+o(1)) \pi\sqrt{2/3}\sqrt{k}> 2 \sqrt{k}$. Set $k=\lceil \log^2 n \rceil$ and observe that 
\[
p(k) \ge e^{2\sqrt{k}} = n^2.
\]
For each partition, $x_1+\dots+x_\ell = k$,  we define an unlabelled forest $F$ of $\ell$ paths with length $x_1,\dots, x_\ell$, respectively. Let $\mathcal{F}$ be the set of linear forests corresponding to $k=\lceil \log^2 n \rceil$. Hence, $|\mathcal{F}| \ge n^2$. Furthermore, observe that all graphs in $\mathcal{F}$ are pairwise non-isomorphic.

Our goal is to embed into $G$ at least $\Omega(\frac{n^2p}{\log^2n})$ different forests from $\mathcal{F}$. In order to do it, we find $\Omega(np)$ edge-disjoint paths, $P_1,\dots,P_m$, each of length $\Omega(n)$. The existence of such paths for $p$ satisfying $\Omega(\frac{\log n}{n}) = p \le 1$ follows, for example, from results about packing of Hamilton cycles in random graphs due to Frieze and Krivelevich~\cite{FK2012}, Knox, K\"{u}hn and Osthus~\cite{FKO2015}, and Krivelevich and Samotij~\cite{KS2012}. However, for smaller $p$ we need to provide a different argument.

\begin{prop}
There is an absolute constant $C>0$ such that if $p\ge \frac{C}{n}$ and $G=G(n,p)$, then  a.a.s.~$G$ contains $\Omega(np)$ edge-disjoint paths, each of length $\Omega(n)$.
\end{prop}

\begin{proof}
The proof is based on the depth first search algorithm (DFS) and it is a variant of the previous approach taken by Dudek and Pra\l{}at~\cite{DP2015}. Ben-Eliezer, Krivelevich and Sudakov~\cite{BKS2012} were the first to successfully apply the DFS algorithm for random graphs.

Let $G=(V,E)=G(n,p)$, where $np \ge C$ for some sufficiently large constant $C>0$. Set $m=np$. We will find greedily edge-disjoint paths $P_1,\dots,P_m$ each of length $cn$ for some sufficiently small constant $c>0$. Assume that we already found $P_1,\dots,P_{i}$ for some $i\le m-1$. Now we show how to find a path $P_{i+1}$ which is edge-disjoint with $P_1,\dots,P_{i}$. 

We perform the following algorithm on $G_{i+1}=(V,E_{i+1})$, where $E_{i+1} = E \setminus (E(P_1\cup\dots\cup E(P_i))$. Assume by contradiction that there is no path of length $cn$ in~$G_{i+1}$.

Let $v_1$ be an arbitrary vertex of $G_{i+1}$, let $P=(v_1)$, $U = V(G) \setminus \{v_1\}$, and $W = \emptyset$. If there is an edge from $v_1$ to $U$ (say from $v_1$ to $v_2$), we extend the path as $P=(v_1,v_2)$ and remove $v_2$ from $U$. We continue extending the path $P$ this way for as long as possible. It might happen that we reach the point of the process in which $P$ cannot be extended, that is, there is a path from $v_1$ to $v_k$ (for some $k \le cn$) and there is no edge from $v_k$ to $U$. If this is the case, $v_k$ is moved to $W$ and we try to continue extending the path from $v_{k-1}$, perhaps reaching another critical point in which another vertex will be moved to $W$, etc. If $P$ is reduced to a single vertex $v_1$ and no edge to $U$ is found, we move $v_1$ to $W$ and simply restart the process from another vertex from $U$, again arbitrarily chosen. 

An obvious but important observation is that during this algorithm there is never an edge between $U$ and $W$, i.e., $e_{G_{i+1}}(U,W)=0$. Moreover, in each step of the process, the size of $U$ decreases by 1 or the size of $W$ increases by 1. Hence, at some point of the process both $U$ and $W$ must have equal size, namely, $|U|=|W|=(n-cn)/2 = \frac{1-c}{2}n$ and $e_{G_{i+1}}(U,W)=0$. We show that this cannot happen.

We will argue now that a.a.s.~$e(U,W)>0$. First observe that in $G=G(n,p)$ we have
\[
\E[e_G(U,W)] = \left(\frac{1-c}{2}n\right)^2 p = \left(\frac{1-c}{2}\right)^2 pn^2 =:\mu
\]
and Chernoff's bound together with the union bound yields
\begin{align*}
\Pr\left(\exists\,U,W: e_G(U,W) \le \frac{1}{2}\mu  \right)
&\le \binom{n}{\frac{1-c}{2}n}^2 \exp\left\{-\frac{\mu}{8}\right\}
\le \left(\frac{2e}{1-c}\right)^{(1-c)n} \exp\left\{-\frac{\mu}{8}\right\}\\
&= \exp\left\{(1-c)n\left( \log\left(\frac{2e}{1-c}\right)  - \frac{1-c}{32} pn \right)  \right\} = o(1),
\end{align*}
for sufficiently small $c$, since $pn\ge C$ and $C$ is arbitrarily large. Consequently, in $G_{i+1}$ the number of edges between $U$ and $W$ satisfies a.a.s.
\[
e_{G_{i+1}}(U,W) \ge \frac{1}{2} \mu - icn
\ge \frac{1}{2} \left(\frac{1-c}{2}\right)^2 pn^2 - cmn
= pn^2\left( \frac{(1-c)^2}{8} - c\right) > 0
\]
for sufficiently small $c$, which is a contradiction, since $e_{G_{i+1}}(U,W)=0$. 
\end{proof}

Now we embed as many as possible forests from $\mathcal{F}$ into each $P_i=(v_{i,1}, v_{i,2},\dots,v_{i, cn+1})$. Let $F$ be the first linear forest that consists of paths $Q_1,\dots,Q_\ell$ of length $x_1,\dots,x_\ell$, respectively. Clearly, $|V(F)| = (x_1+1)+\dots+(x_\ell+1) = k+\ell \le 2k$.
We find a copy $F$ in $P_i$ following the order of vertices of $P_i$. That means that vertices $(v_{i,1},\dots,v_{i,x_1})$ will induce a copy of $Q_1$, vertices $(v_{i,x_1+1},\dots,v_{i,x_1+x_2})$ a copy of $Q_2$, etc. Once $F$ is embedded we embed a new forest starting at vertex $v_{i,k+\ell+1}$.
Since no forest in $\mathcal{F}$ has more than $2k$ vertices, we will be able to pack $\Omega(\frac{n}{k}) = \Omega(\frac{n}{\log^2 n})$ forests on each $P_i$ yielding altogether $\Omega(np)\cdot\Omega(\frac{n}{\log^2 n}) = \Omega(\frac{n^2p}{\log^2 n})$ embedded forests into $G$.

Finally, we color each embedded forest with a different color and the leftover edges $E(G)\setminus (E(P_1)\cup\dots\cup E(P_m))$ get an extra new color. To see that no two color classes are isomorphic it suffices to observe that the number of edges in $|E(G)\setminus (E(P_1)\cup\dots\cup E(P_m))| = \Omega(\frac{n^2p}{\log^2 n}) \gg \log^2 n$.

\section{Proof of Theorem~\ref{thm:main_very_sparse}}

\subsection{Some preliminary remarks}
Assume that $p$ satisfies
\[
n^{-\frac{k}{k-1}} \ll p \ll n^{-\frac{k+1}{k}}.
\]
It is well known that a.a.s.~the number of components (i.e. trees) of order $i\ge 2$ is $\Theta(n^ip^{i-1})$ and there is no component of order~$k+1$. 
When $k=2$, $\tau(G) = O(\sqrt{m}) = O(np^{1/2})$. Indeed, since the graph is is just a matching we form non-isomorphic color classes by letting each color class have a different number of edges.

Now we show that for $k=3$, we also have $\tau(G) = O(np^{1/2})$. Each color class can be represented as a vector $(i,j)$, where $i$ is the number of $P_2$ and $j$ is the number of $P_3$ (for this range of $p$ these two types are the only components). We consider two types: $(*,0)$ and $(*, \ge\!\!1)$, where ``$*$'' denotes any non-negative integer and ``$\ge\!\!1$'' any positive integer.
The first type can have at most $O(np^{1/2})$ color classes (as in case $k=2$). The second type consists of color classes that contain at least one copy of $P_3$. Let $g(i)$ be the number of color classes of the form $(i,\ge\!\!1)$. Now notice that the number of copies of $P_3$ associated with color classes $(i,\ge\!\!1)$ is at least $1+2+\cdots+g(i) \ge \frac 12 g(i)^2$. Therefore, since the total number of $P_3$ is $O(n^3p^2)$, we obtain the following bound
\[
O(n^3p^2) = \frac12 [ g(1)^2+\dots + g(\ell)^2] \ge \frac12 [ g(1)+\dots + g(\ell)]
\]
Hence, the number of color classes is at most
\[
O(np^{1/2}) + g(1)+\dots + g(\ell) = O(np^{1/2} + n^3p^2) = O(np^{1/2}),
\]
since $n^3p^2 \ll np^{1/2}$ due to $p \ll n^{-4/3}$.

For $k\ge 4$ the situation is much more complicated.

\subsection{Motivation}

The motivation for our proof of the upper bound comes from the $\ell = 3$ case (this occurs for the first time when $k=4$). This can be solved using the method of Lagrange multipliers. In our discussion below, we provide a sketch of the procedure and how it motivates our rigorous argument, with the aim of providing insight into the techniques used for larger values of $k$. 

When $\ell =3$, we consider color classes of the form $(*,\ge\!\!1,0)$, where the first coordinate corresponds to the number of $P_2$ and the second to the number of $P_3$. One can check that the color classes of such form dominate the rest ones (using an argument similar to that used in the $k=3$ case). 
We will use $x_t$ to represent the number of $(t,\ge\!\!1,0)$ classes within a given edge coloring. This will lead to the following optimization problem:
Maximize 
\[
\sum_{t=0}^m x_t
\]
subject to 
\[
\sum_{t=0}^m tx_t \le A = O(n^2 p), \qquad \sum_{t=0}^{m} x_t^2 \le B = O(n^3 p^2), \qquad x_t \ge 0
\]
for some positive numbers $A, B$ arising from constraints in $G(n,p)$ on the number of trees of order two and three in this sparse regime. 
In particular, we get the second constraint, since for any $t$, the number of $P_3$ associated with $(t, \ge\!\!1, 0)$ is at least $1+2+\cdots+x_t \ge \frac12x_t^2$, which must be at most $O(n^3p^2)$ when summing over all $t$.

The gradient of our objective function is $(1, \ldots, 1)$, while the gradients of our two constraint functions are $(0,1, 2, \ldots, \ell)$ and $(2x_0,2x_1, 2x_2, \ldots, 2x_\ell)$, respectively. So the method of Lagrange multipliers guarantees that there is an optimal solution such that
\[
(1, \ldots, 1) = \lambda_1 (0,1, 2, \ldots, \ell) + \lambda_2 (2x_0,2x_1, 2x_2, \ldots, 2x_\ell)
\]
for some constants $\lambda_1, \lambda_2$.
We remark that we can assume that $x_1\ge x_2\ldots\ge x_m$ using the rearrangement inequality (applied to the first constraint). 
Therefore, for each $t$ we have 
\[
x_t = \frac{1-\lambda_1 t}{2\lambda_2}.
\]
Thus, we have
$$
x_t = r - st \qquad \textrm{where}\qquad r :=  \frac{1}{2\lambda_2} \qquad \textrm{and}\qquad s := \frac{\lambda_1}{2\lambda_2}.
$$
Since $x_0 \ge x_1 \ge x_2 \ldots \ge x_m \ge 0$ we must have that $s \ge 0$ and $r \ge s m$. Now we have
\[
A = \sum_{t=0}^m t(r-st) \sim \frac12 rm^2 - \frac13 sm^3 \ge \frac 16 rm^2,
\]
\[
B= \sum_{t=0}^m (r-st)^2 \sim r^2 m - rsm^2 + \frac13 s^2 m^3 = \frac{\rbrac{3 - 3\cdot \frac{s m}{r} + \rbrac{\frac {sm}{r}}^2}r^2 m}{3} \ge \frac14 r^2 m
\]
(by minimizing the quadratic in the numerator),
and our objective function is
\[
\sum_{t=0}^m (r-st) \lesssim rm = O\rbrac{(AB)^{1/3}} = O\rbrac{n^{5/3}p},
\]
which is the desired upper bound. Here the important observation is that $x_t$ takes the form $r -s t$ for some positive constants $r,s$, in this case motivated by an application of the method of Lagrange multiplies. We will use this idea in the more general proof below.

\subsection{Upper bound} Let $k\ge 2$ be given and set
\[
\ell = \left\lfloor\frac{\sqrt{8k-7}+1}{2}\right\rfloor.
\]

Notice that each color class can be represented as a vector $(t_1,t_2,\dots,t_{k-1})$, where $t_i$ denotes the number of isolated trees of order $i+1$. Two such color classes would only be isomorphic if they corresponded to exactly the same choice of the numbers~$t_i$. We also assume that not all $t_i$ are equal to zero simultaneously. We will actually do this in such a way that we have exactly one color class for each choice of the $t_i$. Technically speaking one vector $(t_1,t_2,\dots,t_{k-1})$ can represent  two different non-isomorphic classes if for some $i$, $t_i$ trees in one class are not isomorphic to $t_i$ trees in the second class. But since the number of non-isomorphic trees of order $i$  is at most $i^{i-2}$, we get that one vector can represent at most $O_k(1)$ color classes.

First we show that the number of color classes of the form $(\underbrace{*,\dots,*}_{\ell-2 \text{ times}},\ge\!\!1,\underbrace{0,\dots,0}_{k-\ell \text{ times}})$ is significantly larger than the number of the remaining classes (satisfying the certain constraints). In other words, this will imply that trees of order at least $\ell+1$ have negligible contributions to the number of non-isomorphic color classes.  

Consider color classes of the form $(\underbrace{*,\dots,*}_{i-1\text{ times}},\ge 1,0,\dots,0)$ for $\ell \le i \le k-2$. 
Define $g(t_1,\dots,t_{i-1})$ as the number of classes $(t_1,\dots,t_{i-1},\ge 1,0\dots,0)$. Now notice that the number of trees of order $i+1$ satisfies
\[
O(n^{i+1}p^i) = \sum_{t_1,\dots,t_{i-1}} g(t_1,\dots,t_{i-1})^2
\ge \sum_{t_1,\dots,t_{i-1}} g(t_1,\dots,t_{i-1}).
\]
(Recall that the $i$th coordinate corresponds to the number of isolated trees of order $i+1$.)

We will show that $n^{i+1}p^i \ll n^{\frac{(\ell+2)(\ell-1)}{2\ell}} p^{\frac{\ell-1}{2}}$ for any $\ell \le i \le k-2$. Since the L-H-S is maximized for $i=\ell$, it suffices to  show that $n^{\ell+1}p^\ell \ll n^{\frac{(\ell+2)(\ell-1)}{2\ell}} p^{\frac{\ell-1}{2}}$, which is equivalent to
$n^{\frac{\ell^2+\ell+2}{2\ell}} p^{\frac{\ell+1}{2}} \ll 1$. Furthermore, since $p \ll n^{-\frac{k+1}{k}}$, it is enough to show that 
$n^{\frac{\ell^2+\ell+2}{2\ell}} n^{-\frac{(k+1)(\ell+1)}{2k}} \le 1$,
which is equivalent to $\frac{\ell^2+\ell+2}{2\ell} -\frac{(k+1)(\ell+1)}{2k} \le 0$
and finally to $-\ell^2-\ell+2k\le 0$. The latter follows from the following claim.
\begin{claim}
Let $k\ge 2$ be an integer and $\ell = \lfloor\frac{\sqrt{8k-7}+1}{2}\rfloor$. Then,
\[
-\ell^2-\ell+2k\le 0.
\]
\end{claim}
\begin{proof}
First observe that the roots of the quadratic function $-x^2-x+2k$ are $-\frac{1}{2} \pm \frac{\sqrt{8k+1}}{2}$.
Thus, it suffices to show that
\[
-\frac{1}{2} + \frac{\sqrt{8k+1}}{2} \le \ell.
\]

We will consider two cases. For the first one, assume that $k=\frac{a(a+1)}{2}$ for some positive integer $a$. Hence, 
\[
-\frac{1}{2} + \frac{\sqrt{8k+1}}{2} = -\frac{1}{2} + \frac{\sqrt{(a+1)^2}}{2} = a.
\]
Also since $4a^2+4a-7 \ge (2a-1)^2$, we get
\[
\ell = \left\lfloor\frac{\sqrt{8k-7}+1}{2}\right\rfloor = \left\lfloor\frac{\sqrt{4a^2+4a-7}+1}{2}\right\rfloor
\ge \left\lfloor\frac{\sqrt{(2a-1)^2}+1}{2}\right\rfloor = a,
\]
as required.

Now assume that $\frac{a(a+1)}{2} + 1\le k \le \frac{(a+1)(a+2)}{2}-1$. Thus,
\[
\ell = \left\lfloor\frac{\sqrt{8k-7}+1}{2}\right\rfloor \ge 
\left\lfloor\frac{\sqrt{8(\frac{a(a+1)}{2} + 1)-7     }+1}{2}\right\rfloor
= \left\lfloor\frac{\sqrt{(2a+1)^2}+1}{2}\right\rfloor = a+1
\]
and
\[
-\frac{1}{2} + \frac{\sqrt{8k+1}}{2} \le 
-\frac{1}{2} + \frac{\sqrt{8\frac{(a+1)(a+2)}{2}+1}}{2}
= -\frac{1}{2} + \frac{\sqrt{(2a+3)^2}}{2} = a+1,
\]
finishing the proof of the claim.
\end{proof}

For fixed $k \ge 4$, it suffices to show that $|\mathcal C| = O(n^{\frac{(\ell+2)(\ell -1)}{2\ell}} p^{\frac{\ell -1}{2}})$.  
Let $x: \mathbb R^{\ell - 2} \rightarrow \mathbb R$, denoted  $x(t_1, \ldots, t_{\ell-2}):= x_{t_1, \ldots, t_{\ell-2}}$, be  the number of classes $$(t_1, t_2,\ldots, t_{\ell-2}, \ge 1, 0, \ldots,0, 0),$$ where there are $k-\ell$ zeros at the end. 
If we let $t = (t_1, t_2, \ldots, t_{\ell-2})$ and define
$$
X := \{t \in \mathbb N^{\ell-2}: 0 \le t_i \le m \textrm{ for all } i\},
$$
then we wish to maximize
$$
F := \sum_{t\in X} x_{t},
$$
subject to the constraints
\begin{align*}
&\sum_{t \in X}t_ix_{t} \le A_i = O(n^{i+1}p^{i}) \textrm{ for all } 1 \le i \le \ell - 2,\\
&\sum_{t \in X} x_{t}^2 \le B = O(n^{\ell}p^{\ell-1}),\textrm{ and}\\
&x_t \ge 0.
\end{align*}
Define
$$
r_i = (np)^{\ell-1-i}\textrm{ for all } 1 \le i \le \ell - 2 \qquad \textrm{and} \qquad s = n^{(\ell^2-\ell+2)/(2\ell)}p^{(\ell-1)/2}.
$$
Then the following estimate holds: 
\begin{align}
B + 2\sum_{i=1}^{\ell-2} r_iA_i - 2s F 
&\ge \sum_{t\in X} \left(x_t^2 + 2\sum_{i=1}^{\ell-2} r_it_ix_t -2sx_t\right)\nonumber\\
&= \sum_{t\in X} \left(x_t - \left(s-\sum_{i=1}^{\ell-2} r_it_i\right)\right)^2 - \sum_{t\in X} \left(s-\sum_{i=1}^{\ell-2} r_it_i\right)^2. \label{eq:tX}
\end{align}
Let 
$$
S := \left\{t\in X: s-\sum_{i=1}^{\ell-2} r_it_i \ge 0 \right\} \subseteq [0, s/r_1] \times  [0, s/r_2] \times \ldots \times [0, s/r_{\ell - 2}]:=Q.
$$
The last inclusion is clear since for all $i$ if we have $t\in S$ then $s \ge \sum_{i=1}^{\ell-2} r_it_i \ge r_it_i$, so $0 \le t_i \le s/r_i$ for all $i$.
Now since the terms of the first sum in \eqref{eq:tX} are nonnegative, we have 
$$
\sum_{t\in X} \left(x_t - \left(s-\sum_{i=1}^{\ell-2} r_it_i\right)\right)^2 \ge \sum_{t\in X\setminus S} \left(x_t - \left(s-\sum_{i=1}^{\ell-2} r_it_i\right)\right)^2 .
$$
Note that when $t \in X\setminus S$ then $s-\sum_{i=1}^{\ell-2} r_it_i \le 0$ and $x_t \ge 0$, so we have 
$$
\left(x_t - \left(s-\sum_{i=1}^{\ell-2} r_it_i\right)\right)^2 \ge \left(s-\sum_{i=1}^{\ell-2} r_it_i\right)^2 
$$
and
\begin{align}
\sum_{t \in X \setminus S } & \left(x_t - \left(s-\sum_{i=1}^{\ell-2} r_it_i\right)\right)^2 - \sum_{t\in X} \left(s-\sum_{i=1}^{\ell-2} r_it_i\right)^2\nonumber\\
&\ge \sum_{t \in X \setminus S }  \left(s-\sum_{i=1}^{\ell-2} r_it_i\right)^2 - \sum_{t\in X} \left(s-\sum_{i=1}^{\ell-2} r_it_i\right)^2\nonumber\\
&\ge - \sum_{t \in S }  \left(s-\sum_{i=1}^{\ell-2} r_it_i\right)^2
\ge - \sum_{t \in Q }  s^2 \nonumber \\
&= -s^2 \prod_{i=1}^{\ell-2} \frac{s}{r_i} \ge -\ell^2  s^\ell (np)^{-(\ell-1)(\ell-2)/2}. \label{eq:tX-simp}
\end{align}
Then \eqref{eq:tX} and \eqref{eq:tX-simp} imply that
$$
F \le\frac{B + 2\sum_{i=1}^{\ell-2} r_iA_i + \ell^2  s^\ell (np)^{-(\ell-1)(\ell-2)/2} }{2s} =  \frac{O(n^{\ell}p^{\ell-1})}{2n^{(\ell^2-\ell+2)/(2\ell)}p^{(\ell-1)/2}} = O\left(n^{\frac{(\ell+2)(\ell -1)}{2\ell}} p^{\frac{\ell -1}{2}}\right),
$$
as desired.

\subsection{Lower bound}
By using Chebyshev's inequality one can show that the number of isolated paths $P_i$ is at least $c_i n^i p^{i-1}$ for some positive constant $c_i$, where $2\le i\le k$. 

We will consider the following set $\mathcal{C}$ of color classes, each of the form
\[
\{t_2P_2,t_3P_3,\dots,t_\ell P_\ell\},
\]
where $0\le t_i \le \xi_i$,
\[
\ell = \left\lfloor\frac{\sqrt{8k-7}+1}{2}\right\rfloor,
\]
and
\[
\xi_i +1= \frac{c_in^ip^{i-1}}{n^{\frac{(\ell+2)(\ell-1)}{2\ell}} p^{\frac{\ell-1}{2}}}.
\] 
In other words, a color class will be the vertex-disjoint union of paths. Two such color classes would only be isomorphic if they corresponded to exactly the same choice of the numbers~$t_i$. We also assume that no all $t_i$ are equal to zero simultaneously. We will actually do this in such a way that we have exactly one color class for each choice of the $t_i$. 

First observe that due to our choice of $\ell$, each $\xi_i\gg 1$, that means our bounds are well-defined. Indeed,
\begin{align*}
\xi_i +1 \ge \xi_\ell+1 = \frac{c_\ell n^\ell p^{\ell-1}}{n^{\frac{(\ell+2)(\ell-1)}{2\ell}} p^{\frac{\ell-1}{2}}}
&= c_\ell n^{\ell - \frac{(\ell+2)(\ell-1)}{2\ell}} p^{\frac{\ell-1}{2}}\\
&\gg c_\ell n^{\ell - \frac{(\ell+2)(\ell-1)}{2\ell}} n^{-\frac{k}{k-1}\cdot \frac{\ell-1}{2}}
=c_\ell n^{\frac{1}{\ell}-\frac{\ell-1}{2(k-1)}}.
\end{align*}
Now observe that $h(\ell) = \frac{1}{\ell}-\frac{\ell-1}{2(k-1)}$ is a decreasing function for $2\le \ell\le k$ and the only positive root is $\frac{\sqrt{8k-7}+1}{2}$. Thus, $\xi_i\gg 1$ for any $2\le i\le \ell$.

Set $c=\prod_{i=2}^\ell c_i$. We may assume that $c\le 1$. Clearly,
\[
|\mathcal{C}| = \prod_{i=2}^\ell (\xi_i +1)-1 = \frac{ cn^{\frac{(\ell+2)(\ell-1)}{2}} p^{\frac{\ell(\ell-1)}{2}}  }{n^{\frac{(\ell+2)(\ell-1)^2}{2\ell}} p^{\frac{(\ell-1)^2}{2}}}-1
= cn^{\frac{(\ell+2)(\ell-1)}{2\ell}} p^{\frac{\ell-1}{2}}-1
=\Omega\left(n^{\frac{(\ell+2)(\ell-1)}{2\ell}} p^{\frac{\ell-1}{2}} \right),
\]
as required.


\end{document}